%% file: A_Flanking_Pattern_in_a_Sum-of-Divisors_Congruence.tex
\documentclass[11pt]{amsart}
\usepackage[usenames,dvipsnames]{xcolor}
\usepackage{geometry}
\usepackage{amsmath,amssymb,amsthm}
\usepackage{enumerate}
\usepackage{aliascnt}
\usepackage{hyperref}
	\hypersetup{colorlinks=true, citecolor=blue, linkcolor=BrickRed, urlcolor=OliveGreen, pdfstartview={FitH}}
	
	\newtheorem{theorem}{Theorem}[section]

	\newaliascnt{lemma}{theorem}
	\newtheorem{lemma}[lemma]{Lemma}
	\aliascntresetthe{lemma}

	\newaliascnt{corollary}{theorem}
	\newtheorem{corollary}[corollary]{Corollary}
	\aliascntresetthe{corollary}

	\newaliascnt{proposition}{theorem}
	\newtheorem{proposition}[proposition]{Proposition}
	\aliascntresetthe{proposition}

	\theoremstyle{definition}
	\newaliascnt{definition}{theorem}
	\newtheorem{definition}[definition]{Definition}
	\aliascntresetthe{definition}

	\newaliascnt{remark}{theorem}
	\newtheorem{remark}[remark]{Remark}
	\aliascntresetthe{remark}

	\newaliascnt{example}{theorem}
	\newtheorem{example}[example]{Example}
	\aliascntresetthe{example}

	\usepackage[capitalize,noabbrev]{cleveref}
	\crefname{theorem}{Theorem}{Theorems}
	\Crefname{theorem}{Theorem}{Theorems}
	\crefname{lemma}{Lemma}{Lemmas}
	\Crefname{lemma}{Lemma}{Lemmas}
	\crefname{corollary}{Corollary}{Corollaries}
	\Crefname{corollary}{Corollary}{Corollaries}
	\crefname{proposition}{Proposition}{Propositions}
	\Crefname{proposition}{Proposition}{Propositions}
	\crefname{definition}{Definition}{Definitions}
	\Crefname{definition}{Definition}{Definitions}
	\crefname{remark}{Remark}{Remarks}
	\Crefname{remark}{Remark}{Remarks}
	\crefname{example}{Example}{Examples}
	\Crefname{example}{Example}{Examples}

	\newcommand{\Z}{\mathbb{Z}}
	\newcommand{\N}{\mathbb{N}}
	\DeclareMathOperator{\ord}{ord}
	\newcommand{\vtwo}{\nu_2}
	\newcommand\primes{\mathcal{P}}
	\newcommand\set{S}
	\newcommand\term[1]{\emph{#1}}
	\newcommand\card[1]{\left\vert #1\right\vert}
	\let\phi\varphi

 \usepackage{listings}
 
 \definecolor{codegreen}{rgb}{0,0.6,0}
 \definecolor{codegray}{rgb}{0.5,0.5,0.5}
 \definecolor{codepurple}{rgb}{0.58,0,0.82}
 \definecolor{backcolour}{rgb}{0.98,0.98,0.98}
 
 \lstdefinestyle{pythonstyle}{
     language=Python,
     backgroundcolor=\color{backcolour},
     commentstyle=\color{codegreen},
     keywordstyle=\color{magenta},
     numberstyle=\tiny\color{codegray},
     stringstyle=\color{codepurple},
     basicstyle=\ttfamily\footnotesize,
     breakatwhitespace=false,
     breaklines=true,
     captionpos=b,
     keepspaces=true,
     numbers=left,
     numbersep=5pt,
     showspaces=false,
     showstringspaces=false,
     showtabs=false,
     tabsize=4,
     frame=single,
     rulecolor=\color{black}
 }

\title{A Flanking Pattern in a Sum-of-Divisors Congruence}

\author[Scott Duke Kominers]{Scott Duke Kominers}
\address{Harvard Business School; Department of Economics and Center of Mathematical Sciences and Applications, Harvard University; and a16z crypto}
\email{kominers@fas.harvard.edu}
\thanks{I gratefully acknowledge helpful conversations with Ben Golub, Ken Ono, Jesse Shapiro, and especially Noam~D.\ Elkies. Additionally, I used LLMs to assist with some of the computations and coding in the preparation of this article, particularly GPT-5-Pro, Claude-Opus-4.1, and Claude-Sonnet-4.5 (all accessed via Poe with the support of Quora, where I am an advisor). The problem, methods, and eventual written form are my own; and of course any errors remain my responsibility.}


\begin{document}

\begin{abstract}We consider composite $n$ satisfying the congruence \begin{equation*}
n \cdot \sigma_k(n) \equiv 2 \pmod{\phi(n)},
\end{equation*} and show a ``flanking'' structure: $14$ appears in both $\set_{k-1}$ and $\set_{k+1}$ whenever certain values of~$n$ appear in~$\set_k$; and, moreover, $14$ is the only (nontrivial) case of this property. Along the way, we derive a new characterization of the~$n$ that appear in the sets $\set_{k}$.
\end{abstract}

\maketitle

\section{Introduction}\label{sec:intro}

In \cite{gong2010congruence}, we (Gong and the present author) considered the congruence 
\begin{equation}\label{eq:main}
n \cdot \sigma_k(n) \equiv 2 \pmod{\phi(n)},
\end{equation}
where  $\sigma_k$ and $\phi$ denote the \term{sum-of-powers-of-divisors} and \term{totient} functions, respectively, i.e., $\sigma_k(n)=\sum_{d\mid n}d^k$ and $\phi(n)=\card{\{d\in\N:1\leq d\leq n\text{ and }(d,n)=1\}}$. The congruence \eqref{eq:main} to some degree characterizes the primes: it is automatically satisfied by all prime $n$,\footnote{Indeed, when $n$ is prime, we have $\sigma_k(n)= 1^k+n^k$ and $\phi(n)=n-1$; hence, $n\cdot\sigma_k(n)=n(1^k+n^k)=n+n^{k+1}\equiv (1+1) \bmod{(n-1)}=2\bmod{\phi(n)}$.} yet each $k$ admits only finitely many composite solutions. Indeed, in \cite{gong2010congruence}, we showed the following result, building on earlier work of Lescot~\cite{Lescot} and Subbarao~\cite{Subbarao} proving the cases $k\in \{0,1\}$:\footnote{See also \cite{dujella2012variation} for a variation on this congruence.}
\begin{theorem}\label{thm:gk-main}
For any $k\geq 0$, let $\set_k$ be the set of composite $n\in \N$ satisfying~\eqref{eq:main} for a given value of $k$. Then, \begin{enumerate}\item $\set_k\subset 2\primes$, where $\primes$ denotes the set of primes; \item $\set_k$ is finite; and \item the maximal element of $\set_k$ is at most $2^{k+3}+6$.\end{enumerate}
\end{theorem}

We observed moreover that $n\in \set_k$ if and only if $n=2p$ for prime $p$ satisfying
\begin{equation}
p-1 \mid \bigl(2^{k+2}+2\bigr),
\label{eq:gkprimediv}
\end{equation} which implies that any $n\in\N$ that is in $\set_k$ for \textit{some} $k\geq 0$ must be in infinitely many of the sets $\{\set_{k'}\}_{k'\geq 0}$ \cite[Corollary~3]{gong2010congruence}.

In this note, we refine the characterization of the composite $n$ that satisfy~\eqref{eq:main}, and use our refinement to show a ``flanking'' pattern that appears across the sets $\set_k$. 

\subsection{The Exceptional Sets $\set_k$}

\Cref{tab:exceptional_sets} reproduces and slightly corrects and extends a table due to \cite{gong2010congruence} presenting the first few exceptional sets $\set_k$.\footnote{Our original computation \cite[Table 1]{gong2010congruence} somehow missed precisely the largest exceptional $n$ in a few of the sets $\set_k$ with $0\leq k\leq 14$; we correct the computation here, and present the code used to generate~\cref{tab:exceptional_sets} in an Appendix.}

\input{exceptional_sets_table}

Examining~\cref{tab:exceptional_sets} we see immediately that $4$ and $6$ appear in all of the exceptional sets, which is already known to be a general phenomenon (see \cite[Section~3]{gong2010congruence}). Moreover, $22$ appears with regular frequency in all of the sets shown, and when it does appear, it is always ``flanked'' by $14$: i.e., when $22\in \set_k$, we always have $\set_{k-1}\ni 14\in\set_{k+1}$. We show that this pattern is general, and in a certain sense unique to $14$---while being emphatically not unique to~$22$.

\subsection{Flanking in the Exceptional Sets}

\newcommand\dist{\ell}

\begin{definition}
Formally, for $\dist\in\N$ and $k\geq\dist$, we say that $n_*$ \term{flanks $n$ at distance $\dist$ at $k$} if $n\in \set_k$ and we have $n_*\in \set_{k-\dist}$ and $n_*\in \set_{k+\dist}$. If $n_*$ flanks $n$ at distance $\dist$ at $k$ for \textit{every} $k$ for which $n\in \set_k$ (and there is at least one such $k$), then we say that $n_*$ \term{(nontrivially) flanks $n$ at every occurrence}.
\end{definition} 

In the sequel, we prove:
\begin{enumerate}
	\item $14$ flanks $22$ at distance $1$ at every occurrence (a warm-up, shown in~\cref{cor:14-22}).
	\item $14$ is the \textit{only} integer that can flank nontrivially at distance $1$ [other than $4$ and $6$, which do so mechanically because they are in all of the sets $\set_k$] (\cref{cor:only-6-14});
	\item $14$ nontrivially flanks $2p$ at distance $1$ at every occurrence if and only if  $-1 \in \langle 2 \rangle$ in $(\Z/r\Z)^\times$ and the least $t_0$ with $2^{t_0}\equiv -1 \bmod r$ is even (\cref{thm:main}).
\end{enumerate}
Along the way, we give a precise characterization of membership in $\set_k$ (\cref{thm:occurrence-criterion}) and derive a partial characterization of when flanking is possible in general (\cref{lem:flanking-distance-constraint}). 

\section{A Refined Criterion for Membership in $\set_k$}\label{sec:criterion}

\newcommand{\az}{a}
\newcommand{\ax}{b}

We write $\vtwo(n)$ for the $2$-adic valuation of a positive integer $n$ (i.e., the largest $\az$ for which $2^\az \mid n$). For an integer $\ax$ coprime to $m \ge 2$, we write $\ord_m(\ax)$ for the multiplicative order of $\ax$ modulo $m$. When $p$ is an odd prime and $p \equiv 3 \bmod{4}$, we have $\vtwo(p-1) = 1$ and thus $p-1 = 2r$ with $r$ odd.

Our first goal is to clarify the arithmetic of $\set_k$ for $n=2p$.

\begin{proposition}\label{prop:2p-criterion}
Let $p$ be an odd prime and write $p-1 = 2^{\vtwo(p-1)} r$ with $r$ odd. Then $2p\in \set_k$ if and only if\footnote{It is possible also to derive this criterion from~\eqref{eq:gkprimediv}; we give a direct proof here both for completeness and to illustrate a bit of how we work with the congruence~\eqref{eq:main}.}
\[
p \equiv 3 \pmod{4} \quad \text{and} \quad r \mid \bigl(2^{k+1}+1\bigr).
\]
\end{proposition}

\begin{proof}
Using multiplicativity of $\sigma_k$, we have
\[
\sigma_k(2p)=(1+2^k)(1+p^k).
\]
As $\phi(2p)=p-1$, we have $p\equiv 1\bmod \phi(2p)$, so $1+p^k\equiv 2\bmod \phi(2p)$ and $2p\equiv 2\bmod \phi(2p)$. Therefore
\begin{equation}
2p\,\sigma_k(2p)\equiv 2\cdot(1+2^k)\cdot 2 \equiv 4(1+2^k)\pmod{p-1}.
\label{eq:intermediate}
\end{equation}
Bringing~\eqref{eq:intermediate} to~\eqref{eq:main}, we see that $2p\in \set_k$ if and only if
\begin{equation}
4(1+2^k)\equiv 2 \pmod{p-1}\ \iff\ 2\bigl(2^{k+1}+1\bigr)\equiv 0\pmod{p-1}.
\label{eq:intermediate2}
\end{equation}
Writing $p-1=2^\az r$ with $r$ odd and $\az = \vtwo(p-1)$, we need $2^{\az} r \mid 2(2^{k+1}+1)$. Since $2^{k+1}+1$ is odd, we must have $2^\az \mid 2$; hence, $\az \leq 1$. As $p-1$ is even for $p>2$, we must in fact then have $\az = 1$, i.e.\ $p \equiv 3 \pmod{4}$ and $p-1=2r$ with $r$ odd; under this decomposition, the condition~\eqref{eq:intermediate2} requires $r\mid(2^{k+1}+1)$.
\end{proof}

Two immediate corollaries recover and streamline the cases $n=14$ and $n=22$.

\begin{corollary}\label{cor:14-membership}
$14\in \set_k$ if and only if $k$ is even.
\end{corollary}

\begin{proof}
Here with $14=2p$, we have $p=7$, so in the decomposition of~\cref{prop:2p-criterion}, $r=(7-1)/2=3$, and so $14\in \set_k$ if and only if $3\mid (2^{k+1}+1)$. Since $2\equiv -1\bmod 3$, we have $3\mid (2^{k+1}+1)$ if and only if $(-1)^{k+1}\equiv -1$, i.e., if and only if $k$ is even.
\end{proof}

\begin{corollary}\label{cor:22-membership}
$22\in \set_k$ if and only if $k\equiv 1\bmod 4$.
\end{corollary}

\begin{proof}
Here with $22=2p$, we have $p=11$, so in the decomposition of~\cref{prop:2p-criterion}, $r=5$, and so $22\in \set_k$ if and only if $5\mid (2^{k+1}+1)$. Since $\ord_5(2)=4$, the condition is equivalent to $k+1\equiv 2\bmod 4$, i.e., $k\equiv 1\bmod 4$.
\end{proof}

\Cref{cor:14-membership,cor:22-membership} show that the $14$--$22$--$14$ flanking relationship observed in~\cref{tab:exceptional_sets} is robust, and in fact appears with period $4$ in the $\set_k$: $22\in \set_k$ if and only if $k\equiv 1\bmod 4$, and $14\in \set_k$ if and only if $k\equiv 0 \text{ or } 2\bmod 4$, so every time $22$ appears in $\set_k$, it is flanked by $\set_{k-1}\ni 14\in\set_{k+1}$.

\begin{corollary}\label{cor:14-22}
$14$ flanks $22$ at distance $1$ at every occurrence.
\end{corollary}

Note also that the fact that we have $6\in \set_k$ for every $k\ge 0$ follows from~\cref{prop:2p-criterion} with $p=3$ (then $r=1$).

\section{A General Constraint on Flanking}

Building on~\cref{prop:2p-criterion}, we obtain a general necessary condition for distance-$\dist$ flanking.

\begin{lemma}\label{lem:flanking-distance-constraint}
Let $p_1\neq p_2$ be odd primes and suppose $2p_1$ (nontrivially) flanks $2p_2$ at distance $\dist\geq 1$ and set $r_1=(p_1-1)/2$. Then:
\[
r_1 \mid \bigl(2^{2\dist}-1\bigr).
\]
\end{lemma}

\begin{proof}
Since the flanking is nontrivial by hypothesis, there exists $k$ with $2p_2\in \set_k$ and $\set_{k-\dist} \ni 2p_1 \in \set_{k+\dist}$. Then, by~\cref{prop:2p-criterion}, we have $r_1=(p_1-1)/2$ odd and
\[
r_1 \mid \bigl(2^{k-\dist+1}+1\bigr)\quad\text{and}\quad r_1 \mid \bigl(2^{k+\dist+1}+1\bigr).
\] 
We now consider the identity
\begin{equation}
\bigl(2^{k+\dist+1}+1\bigr) - 2^{2\dist}\bigl(2^{k-\dist+1}+1\bigr) = 1-2^{2\dist}:
\label{eq:identity}
\end{equation}
since $r_1$ divides both terms on the left of~\eqref{eq:identity}, it must divide their difference. Hence, we see that $r_1\mid(2^{2\dist}-1)$.
\end{proof}

For distance $\dist=1$,~\cref{lem:flanking-distance-constraint} pins down the only possible flankers.

\begin{corollary}\label{cor:only-6-14}
For odd primes $p_1\neq p_2$, if $2p_1$ flanks $2p_2$ at distance $1$ at any $k$, then $p_1\in\{3,7\}$. In other words, the only composite integers $n>4$ that can flank nontrivially at distance $1$ are $6$ and $14$.
\end{corollary}

\begin{proof}
With $\dist=1$,~\cref{lem:flanking-distance-constraint} gives $r_1\mid 3$. Since $r_1=(p_1-1)/2$ is a positive integer, we must have $r_1\in\{1,3\}$, i.e., $p_1\in\{3,7\}$.
\end{proof}

As $4\in \set_k$ and $6\in \set_k$ for all $k$, we see that $4$ and $6$ automatically flank every integer that ever appears in one of the exceptional sets $\set_k$. Thus, if we abuse terminology slightly for a moment, we conclude that the only ``nontrivial'' nontrivial distance-1 flanker is $14$. 

\section{When does $14$ flank $2p$ at distance $1$?}

We now classify exactly which $2p$ are nontrivially flanked by $14$.

\begin{theorem}\label{thm:14-flanking-characterization}
Let $p\geq 5$ be an odd prime with $p \equiv 3 \pmod{4}$ and write $p-1=2r$ with $r>1$ odd. Set $m := \ord_r(2)$. Then $14$ nontrivially flanks $2p$ at distance $1$ at some $k$ if and only if there exists an even integer $t_0 \ge 2$ such that
\[
2^{t_0} \equiv -1 \pmod{r}.
\]
In this case, for the minimal such $t_0$, we have $m = 2t_0$, and the least $k \ge 0$ with $r \mid (2^{k+1}+1)$ is
\[
k_{\min} := t_0 - 1,
\]
and all solutions are given by $$k \equiv t_0 - 1 \pmod{2t_0}.$$
Moreover, if $r$ is prime, then the condition on $r$ is equivalent to the requirement that $4 \mid \ord_r(2)$.
\end{theorem}

\begin{proof}
By~\cref{prop:2p-criterion}, membership of $2p$ in $\set_k$ is equivalent to the requirement that $r \mid (2^{k+1}+1)$, i.e., \begin{equation}
2^{k+1} \equiv -1 \pmod{r}.
\label{eq:mod_condition}
\end{equation} If~\eqref{eq:mod_condition} has any solution, then the least positive solution is $k_{\min} = t_0-1$, where $t_0$ is the least positive integer with $2^{t_0} \equiv -1 \bmod{r}$, and the full set of solutions is
\[
K(r,t_0):=\{k: k\equiv t_0-1 \bmod{\ord_r(2)}\}.
\]

By~\cref{cor:14-membership}, $14\in \set_j$ if and only if $j$ is even. Therefore, for any $k\in K(r,t_0)$, we have $14\in \set_{k-1}$ and $14\in \set_{k+1}$ if and only if $k-1$ and $k+1$ are even, i.e., if and only if $k$ is odd. Because all $k\in  K(r,t_0)$ share the same parity as $t_0-1$, we see that $14$ flanks $2p$ precisely when $t_0-1$ is odd; equivalently, when $t_0$ is even.

Now, if such an even $t_0$ exists, then $\ord_r(2)$ divides $2t_0$ but not $t_0$; hence, $\ord_r(2)=2t_0$ and the progression becomes $$k \equiv t_0-1 \pmod{2t_0}.$$

Finally, if $r$ is prime, then $(\mathbb{Z}/r\mathbb{Z})^\times$ is cyclic of even order $r-1$. The element $-1$ has order~$2$, and $2^{t_0} \equiv -1 \pmod{r}$ holds if and only if $\ord_r(2)$ is even with $t_0 = \ord_r(2)/2$. The condition that $t_0$ is even is then equivalent to the condition that $4 \mid \ord_r(2)$.
\end{proof}

\begin{remark}
Note that in the case $p=3$ we have $2p=6$ and $p-1=4$, yielding $r=1$ in the setup of~\cref{thm:14-flanking-characterization}, under which the various conditions $\bmod r$ become vacuous. And indeed, as we have already noted, we have $6\in \set_k$ for all $k$, while $14\in \set_k$ exactly for even $k$ (\cref{cor:14-membership}), so $14$ does not flank $6$. This is why the statement of~\cref{thm:14-flanking-characterization} excludes $p=3$ by assuming $p\geq 5$.
\end{remark}

\begin{corollary}\label{cor:mod8}
If $r=(p-1)/2$ is prime and $14$ nontrivially flanks $2p$ at distance $1$, then $p \equiv 3 \bmod{8}$.
\end{corollary}

\begin{proof}
If $r$ is prime and $14$ flanks $2p$, then~\cref{thm:14-flanking-characterization} implies that $4 \mid \ord_r(2)$; hence, $4 \mid (r-1)$. Thus, $r \equiv 1 \bmod{4}$; as $p-1=2r$, this yields $p \equiv 3 \bmod{8}$, as claimed.
\end{proof}

\subsection{Occurrence and Periodicity of $2p$ in the Sets $\set_k$}

Now, we isolate the appearance pattern of $2p$ in the family $\{\set_k\}_{k\ge 0}$.

\begin{theorem}\label{thm:occurrence-criterion}
Let $p\geq 5$ be an odd prime with $p \equiv 3 \pmod{4}$ and write $p-1=2r$ with $r>1$ odd; let $m:=\ord_r(2)$.
Then the following are equivalent:
\begin{enumerate}[(i)]
\item $2p \in \set_k$ for at least one $k\ge 0$.\label{cond:occ-exists}
\item $-1 \in \langle 2\rangle$ in $(\Z/r\Z)^\times$.\label{cond:occ-subgroup}
\item There exists $t_0\ge 1$ with $2^{t_0}\equiv -1 \bmod r$; for the least such $t_0$ we have $m=2t_0$.\label{cond:occ-t0}
\end{enumerate}
When any one of the conditions \ref{cond:occ-exists}--\ref{cond:occ-t0} holds, the set of all $k\ge 0$ with $2p\in \set_k$ is the infinite arithmetic progression characterized by
\[
k \equiv t_0-1 \pmod{2t_0};\qquad\text{equivalently,}\qquad k \equiv \frac{m}{2}-1 \pmod m.
\]
\end{theorem}

\begin{proof}
By~\cref{prop:2p-criterion}, $2p\in \set_k$ if and only if $2^{k+1}\equiv -1 \bmod r$, which is equivalent to condition \ref{cond:occ-subgroup} and to the existence of $t_0$ as in condition \ref{cond:occ-t0}. If $t_0$ is the least solution of $2^{t_0}\equiv -1 \bmod r$, then $2^{2t_0}\equiv 1$ while $2^{t_0}\not\equiv 1$, so $\ord_r(2)$ divides $2t_0$ but not $t_0$, forcing $m=2t_0$. The solution set in $k$ then follows immediately, proving both periodicity and infinitude.
\end{proof}

\subsection{Full Characterization}

Combining~\cref{thm:14-flanking-characterization,thm:occurrence-criterion} yields a clean trichotomy:

\begin{theorem}\label{thm:main}
Let $p\geq 5$ be an odd prime with $p \equiv 3 \pmod{4}$ and $r = (p-1)/2$. Exactly one of the following holds:
\begin{enumerate}[(A)]
\item $-1 \notin \langle 2 \rangle$ in $(\mathbb{Z}/r\mathbb{Z})^{\times}$: then $2p$ never appears in any $\set_k$.\label{case:never}
\item $-1 \in \langle 2 \rangle$ and the least $t_0$ with $2^{t_0}\equiv -1 \bmod r$ is odd: then $2p$ appears in infinitely many $\set_k$, but is never flanked by $14$ at distance $1$.\label{case:unflanked}
\item $-1 \in \langle 2 \rangle$ and the least $t_0$ with $2^{t_0}\equiv -1 \bmod r$ is even: then $2p$ appears in infinitely many $\set_k$ and is flanked by $14$ at distance $1$ at every occurrence.\label{case:flanked}
\end{enumerate}
In the case that $r$ is prime, case~\ref{case:flanked} occurs if and only if $4 \mid \ord_r(2)$.
\end{theorem}

\begin{remark}
Note that an implication embedded in~\cref{thm:main} is that if $14$ flanks $n$ at distance $1$ at $k$, \textit{for some $k$}, then  $14$ flanks $n$ at distance $1$ \textit{at every occurrence}. That is, if $n\in \set_k$ and $\set_{k-1}\ni 14\in\set_{k+1}$, then there is no $k'$ such that $n\in \set_{k'}$ but $14\notin[\set_{k'-1}\cap\set_{k'+1}]$.
\end{remark}

\begin{proof}[Proof of~\cref{thm:main}]
By~\cref{thm:occurrence-criterion}, either $2p$ never appears (case \ref{case:never}), or it appears in the progression $k\equiv t_0-1\bmod{2t_0}$ for the least $t_0\ge 1$ with $2^{t_0}\equiv -1\bmod r$. By~\cref{thm:14-flanking-characterization}, $14$ flanks $2p$ at distance~$1$ exactly when $t_0$ is even, yielding cases \ref{case:unflanked} and \ref{case:flanked}. The prime-$r$ clause also follows from~\cref{thm:14-flanking-characterization}.
\end{proof}

\begin{example}
For $p=11$ we have $r=(p-1)/2=5$, $\ord_5(2)=4$, so $t_0=2$ and $k_{\min}=1$; indeed $2^{2} \equiv -1 \pmod{5}$ and $14$ flanks $2p=22$. Moreover, $22 \in \set_k$ for all $k \equiv 1 \pmod{4}$.

The same mechanism applies for $p=59$ ($r=29$, $\ord_{29}(2)=28$, $t_0=14$) yielding $2p=118$, and for $p=83$ ($r=41$, $\ord_{41}(2)=20$, $t_0=10$) yielding $2p=166$; each of these appears in infinitely many $\set_k$, flanked by $14$ at every occurrence.
\end{example}

We implemented a \texttt{Python} script to identify primes $p$ satisfying the conditions required for~$2p$ to be nontrivially flanked by $14$ (see~\cref{sec:code} for source code). 

\input{flanked_values_table} 

Among primes $p \leq \maxconsidered$, we found $67$ values of $2p$ that are flanked by $14$; these are presented in~\cref{tab:flanked_values} along with the first $k$ for which they appear in $\set_k$ [indicated in brackets]. By construction, all of the identified values satisfy the criteria in~\cref{thm:14-flanking-characterization}. Consistent with~\cref{cor:mod8}, the associated $p$ are $\equiv 3 \bmod{8}$ when $r$ is prime. By~\cref{thm:main}, each of the values in~\cref{tab:flanked_values} appears in infinitely many sets $\set_k$ and is flanked by $14$ at every occurrence.

\section{Concluding Remarks}

It seems likely that the approaches used here could be extended to at least characterize flanking behavior at distance $\dist$ for small $\dist>1$, but we leave such exploration to future work. More ambitiously, we might hope for a (much) tighter bound on the maximum value in $\set_k$ since while the $2^{k+3}+6$ upper bound given in~\cref{thm:gk-main} is sharp in the case that $2^{k+2}+3$ is prime,\footnote{If $2^{k+2}+3$ is prime, then we see immediately that $2(2^{k+2}+3)\in \set_k$ because~\eqref{eq:gkprimediv} is satisfied.} most of the exceptional $n$ we have identified thus far seem quite small relative to the bound. A tighter bound would enable us to enumerate more of the $\set_k$ and look for further patterns. 

Meanwhile, it seems plausible that there would be infinitely many primes $p$ such that $14$ flanks $2p$ at distance $1$, but unfortunately proving this may be beyond current methods. At least in the case of prime $r$, the condition for case~\ref{case:flanked} of~\cref{thm:main} seems closely related to the infinitude of Sophie Germain/safe primes, which is a long-standing open question (see, e.g., \cite[Section~A7]{guy2004unsolved} and \cite[Section~5.II]{ribenboim2012new}). However, we can at least obtain a conditional result in this direction:

\begin{theorem}[conditional on Hardy--Littlewood prime $k$-tuples]\label{thm:HL-infinite}
Assuming the Hardy--Littlewood prime $k$-tuple conjecture for pairs of linear forms, there exist infinitely many primes $p$ such that $14$ flanks $2p$ at distance $1$ at every occurrence.
\end{theorem}

\begin{proof}[Proof (conditional)]
Consider the admissible pair of linear forms
\[
f_1(n)=8n+5,\qquad f_2(n)=16n+11.
\]
Assuming the Hardy--Littlewood conjecture, there are infinitely many integers $n$ for which $r:=f_1(n)$ and $p:=f_2(n)$ are both prime. For each such $n$ we have $p=2r+1$ (i.e., $p$ is a safe prime) and $r\equiv 5\bmod{8}$; hence, $\left(\frac{2}{r}\right)=-1$ and $\vtwo(r-1)=2$. In a cyclic group of order $r-1$, an element that is not a square has order with full $2$-adic valuation, so $4\mid \ord_r(2)$. By~\cref{thm:14-flanking-characterization}, this implies that $14$ flanks $2p$ at distance $1$ at every occurrence. 
\end{proof}

\bibliographystyle{amsalpha}

\providecommand{\bysame}{\leavevmode\hbox to3em{\hrulefill}\thinspace}
\providecommand{\MR}{\relax\ifhmode\unskip\space\fi MR }
\providecommand{\MRhref}[2]{%
  \href{http://www.ams.org/mathscinet-getitem?mr=#1}{#2}
}
\providecommand{\href}[2]{#2}

\newpage\appendix
	\crefalias{section}{appendix}
	\crefname{appendix}{Appendix}{Appendixes}
	\Crefname{appendix}{Appendix}{Appendixes}
	
\section{Source Code for Generating Tables~\ref{tab:exceptional_sets} and~\ref{tab:flanked_values}}\label{sec:code}
\input{source_code_appendix}

\end{document}

%% file: exceptional_sets_table.tex
\begin{table}[htb]
\centering
\begin{tabular}{c|l}
$k$ & $S_k$ \\ \hline
0 & \{4,6,14\} \\
1 & \{4,6,22\} \\
2 & \{4,6,14,38\} \\
3 & \{4,6\} \\
4 & \{4,6,14,46,134\} \\
5 & \{4,6,22,262\} \\
6 & \{4,6,14\} \\
7 & \{4,6\} \\
8 & \{4,6,14,38\} \\
9 & \{4,6,22,166\} \\
10 & \{4,6,14,2734,8198\} \\
11 & \{4,6\} \\
12 & \{4,6,14\} \\
13 & \{4,6,22,118,454,65542\} \\
14 & \{4,6,14,38,46,134,398,3974,14566,131078\} \\
15 & \{4,6\} \\
16 & \{4,6,14,174766,524294\} \\
17 & \{4,6,22,262,2182\} \\
18 & \{4,6,14\} \\
19 & \{4,6\} \\
20 & \{4,6,14,38\} \\
21 & \{4,6,22\} \\
22 & \{4,6,14,11184814\} \\
23 & \{4,6,691846\} \\
24 & \{4,6,14,46,134,1006,33134,178246\} \\
25 & \{4,6,22,214,3142\} \\
26 & \{4,6,14,38,326,6158,536870918\} \\
27 & \{4,6\} \\
28 & \{4,6,14,2147483654\} \\
29 & \{4,6,22,166,262,10006,130054,1611622,171798694\} \\
30 & \{4,6,14\} \\
31 & \{4,6,2566\} \\
32 & \{4,6,14,38,2734,8198\} \\
33 & \{4,6,22,3814,526342\} \\
34 & \{4,6,14,46,134,1126,1894,344686,489106598\} \\
35 & \{4,6,29446,2634118\} \\
\end{tabular}\medskip
\caption{The exceptional sets $S_k$ for $0 \leq k \leq 35$.}
\label{tab:exceptional_sets}
\end{table}

%% file: flanked_values_table.tex
\newcommand\maxconsidered{10000}

\begin{table}[htb]
\centering
\begin{align*}
\{&22{\scriptsize[1]}, 118{\scriptsize[13]}, 166{\scriptsize[9]}, 214{\scriptsize[25]}, 262{\scriptsize[5]}, 454{\scriptsize[13]}, 502{\scriptsize[49]}, 694{\scriptsize[85]}, 1174{\scriptsize[145]},\\
&2038{\scriptsize[253]}, 2182{\scriptsize[17]}, 2374{\scriptsize[73]}, 2566{\scriptsize[31]}, 2614{\scriptsize[325]}, 2902{\scriptsize[69]}, 3046{\scriptsize[189]},\\
&3142{\scriptsize[25]}, 3238{\scriptsize[201]}, 3622{\scriptsize[89]}, 3814{\scriptsize[33]}, 4054{\scriptsize[45]}, 4918{\scriptsize[613]}, 5638{\scriptsize[351]},\\
&5878{\scriptsize[41]}, 6406{\scriptsize[199]}, 6502{\scriptsize[149]}, 6934{\scriptsize[865]}, 6982{\scriptsize[173]}, 7078{\scriptsize[209]}, 7558{\scriptsize[235]},\\
&7606{\scriptsize[949]}, 7846{\scriptsize[233]}, 7894{\scriptsize[985]}, 8182{\scriptsize[101]}, 8278{\scriptsize[1033]}, 8422{\scriptsize[209]}, 8518{\scriptsize[265]},\\
&8566{\scriptsize[1069]}, 9094{\scriptsize[283]}, 10006{\scriptsize[29]}, 10102{\scriptsize[49]}, 10198{\scriptsize[1273]}, 10774{\scriptsize[1345]},\\
&10966{\scriptsize[1369]}, 11014{\scriptsize[687]}, 11302{\scriptsize[69]}, 12262{\scriptsize[305]}, 12646{\scriptsize[125]}, 13318{\scriptsize[831]},\\
&13558{\scriptsize[241]}, 13654{\scriptsize[1705]}, 14374{\scriptsize[897]}, 14662{\scriptsize[121]}, 14902{\scriptsize[369]}, 15046{\scriptsize[93]},\\
&15286{\scriptsize[381]}, 16294{\scriptsize[1017]}, 16486{\scriptsize[473]}, 16582{\scriptsize[413]}, 16726{\scriptsize[125]}, 17398{\scriptsize[2173]},\\
&17494{\scriptsize[2185]}, 17926{\scriptsize[279]}, 18934{\scriptsize[181]}, 19174{\scriptsize[1197]}, 19606{\scriptsize[545]}, 19702{\scriptsize[489]},\ldots\\
&\}
\end{align*}
\caption{All values $n = 2p$ with $p \leq \maxconsidered$ such that $14$ flanks $n$ at distance~$1$, with first appearance $k_{\min}$ shown in brackets. We found $67$ such values among primes $p \leq \maxconsidered$ with $p \equiv 3 \pmod{{4}}$.}
\label{tab:flanked_values}
\end{table}

%% file: source_code_appendix.tex
%
%

\lstset{style=pythonstyle}
\begin{lstlisting}[caption={\texttt{Python} implementation for computing exceptional sets and values flanked by $14$, as well as this quine of the generator source code.}]
#!/usr/bin/env python3
"""
Exceptional Sets and Flanking Patterns in a Sum-of-Divisors Congruence

This program computes exceptional sets S_k for the congruence
    n * sigma_k(n) == 2 (mod phi(n)),
where sigma_k is the sum of k-th powers of divisors and phi is Euler's totient function.

It also analyzes flanking patterns, showing that 14 flanks certain values at distance 1 in the sequence of exceptional sets.

Implementation notes (matching the present paper's refinements):
- For each k, instead of scanning primes p up to the theoretical bound, we use
  Proposition 2.1: for n = 2p with p an odd prime, we write p-1 = 2r with r odd.
  Then 2p in S_k iff p == 3 (mod 4) and r | (2^(k+1) + 1).
  Thus, for fixed k, we can enumerate r by enumerating divisors of
      M := 2^(k+1)+1,
  and testing p = 2r + 1 for primality.
- For multiplicative orders we use sympy.ntheory (n_order) rather than a custom loop.

The analysis is based on results from [GK10], as well as the results of the present paper.

QED
"""

from __future__ import annotations

from typing import Dict, Iterable, List, Optional, Tuple
import math
import os
import sys

from sympy import divisor_sigma, totient
from sympy.ntheory import factorint, isprime, primerange
from sympy.ntheory.residue_ntheory import n_order

# ============================================================================
# CORE MATHEMATICAL FUNCTIONS
# ============================================================================

def sigma_k(n: int, k: int) -> int:
    """
    Compute sigma_k(n) = sum_{d|n} d^k.

    Uses sympy's divisor_sigma (fast via factorization).
    """
    return int(divisor_sigma(n, k))

def divisors_from_factorint(factors: Dict[int, int]) -> List[int]:
    """
    Generate the (positive) divisors of an integer via factorint dictionary.

    Args:
        factors: dict mapping prime -> exponent, e.g., {3: 2, 5: 1} for 45.

    Returns:
        Sorted list of positive divisors.
    """
    divs = [1]
    for p, e in factors.items():
        p_pows = [p**i for i in range(1, e + 1)]
        divs = [d * pp for d in divs for pp in ([1] + p_pows)]
    return sorted(divs)

def check_membership_2p(p: int, k: int) -> bool:
    """
    Check if 2p is in S_k using the characterization given in Proposition 2.1.

    For odd prime p: write r = (p-1)/2. Then 2p in S_k iff
      p == 3 (mod 4) and r | (2^(k+1) + 1).

    Note: for p = 3 we have r = 1 and the divisibility condition is vacuous
    (and indeed 6 is in S_k for all k).
    """
    if p % 4 != 3:
        return False

    r = (p - 1) // 2
    # Check 2^(k+1) == -1 (mod r); for r=1 Python's pow(..., 1) = 0, and -1 == 0 (mod 1)
    return pow(2, k + 1, r) == (r - 1)

# ============================================================================
# EXCEPTIONAL SET COMPUTATION
# ============================================================================

def find_exceptional_set(k: int) -> List[int]:
    """
    Find all composite n satisfying n * sigma_k(n) == 2 (mod phi(n)).

    Uses the refined structural result (GK10 + Proposition 2.1):
      S_k subset of {4} union {2p : p is an odd prime, p == 3 (mod 4), and r | (2^(k+1)+1) with r=(p-1)/2}.
    So we enumerate r | (2^(k+1)+1) and test p=2r+1 for primality.

    Args:
        k: Non-negative integer index.

    Returns:
        Sorted list of exceptional values for given k.
    """
    max_n = 2 ** (k + 3) + 6
    exceptional: set[int] = set()

    # n = 4 special-case check (kept for robustness)
    n = 4
    phi_n = int(totient(n))
    n_sigma = n * sigma_k(n, k)
    if n_sigma % phi_n == (2 % phi_n):
        exceptional.add(n)

    # Enumerate candidates n = 2p via divisors r of M = 2^(k+1) + 1
    M = 2 ** (k + 1) + 1
    divs_r = divisors_from_factorint(factorint(M))

    for r in divs_r:
        # M is odd, so r is odd,
        # but we guard/confirm.
        if r % 2 == 0:
            continue

        p = 2 * r + 1
        n = 2 * p
        if n > max_n:
            continue
        if isprime(p):
            # For r odd, p = 2r+1 automatically satisfies p == 3 (mod 4),
            # but we keep the logical check explicit.
            if check_membership_2p(p, k):
                exceptional.add(n)

    return sorted(exceptional)

# ============================================================================
# FLANKING ANALYSIS
# ============================================================================

def check_flanking_criterion(p: int) -> Tuple[bool, Optional[int], Optional[int]]:
    """
    Check whether 2p is flanked by 14 at distance 1 using the criteria given in Theorem 4.1.

    For p == 3 (mod 4), write r = (p-1)/2 (odd). Let m = ord_r(2).
    Then:
      - 2^t0 == -1 (mod r) holds iff m is even and 2^(m/2) == -1 (mod r),
        in which case t0 = m/2.
      - 14 flanks 2p at distance 1 iff t0 is even (and r>1).

    Returns:
        (is_flanked, period, k_min) where
          period = 2*t0 (which equals m in the flanking case),
          k_min  = t0 - 1.
    """
    if p % 4 != 3:
        return False, None, None

    r = (p - 1) // 2
    if r <= 1:
        # This is p=3 (r=1), and 14 does not flank 6.
        return False, None, None

    # m = ord_r(2) via sympy.ntheory
    try:
        m = int(n_order(2, r))
    except Exception:
        return False, None, None

    # Existence of t0 with 2^t0 == -1 (mod r) is equivalent to:
    #   m even and 2^(m/2) == -1 (mod r)
    if m % 2 != 0:
        return False, None, None

    t0 = m // 2
    if pow(2, t0, r) != (r - 1):
        return False, None, None

    # Flanking by 14 at distance 1 requires t0 even.
    if t0 % 2 != 0:
        return False, None, None

    k_min = t0 - 1
    period = 2 * t0  # equals m
    return True, period, k_min

def find_flanked_values(max_p: int = 10000) -> List[Tuple[int, int, int]]:
    """
    Find all values 2p (p <= max_p) that are flanked by 14 at distance 1.

    Uses sympy.primerange for prime iteration.
    """
    flanked: List[Tuple[int, int, int]] = []

    print(f"Checking primes up to {max_p} for flanking criterion...")

    checked = 0
    for p in primerange(3, max_p + 1):
        checked += 1
        if checked % 100 == 0:
            print(f"  Checked {checked} primes...")

        is_flanked, period, k_min = check_flanking_criterion(int(p))
        if is_flanked:
            flanked.append((2 * int(p), int(period), int(k_min)))

    print(f"  Total primes checked: {checked}")
    print(f"  Flanked values found: {len(flanked)}")

    return sorted(flanked)

# ============================================================================
# LATEX TABLE GENERATION
# ============================================================================

def generate_exceptional_sets_table(max_k: int = 35) -> str:
    """
    Generate LaTeX table of exceptional sets S_k.
    """
    all_Sk: Dict[int, List[int]] = {}
    for k in range(max_k + 1):
        all_Sk[k] = find_exceptional_set(k)

    lines: List[str] = []
    lines.append("\\begin{table}[htb]")
    lines.append("\\centering")
    lines.append("\\begin{tabular}{c|l}")
    lines.append("$k$ & $S_k$ \\\\ \\hline")

    for k in range(max_k + 1):
        Sk = all_Sk[k]
        set_str = "\\{" + ",".join(map(str, Sk)) + "\\}"
        lines.append(f"{k} & {set_str} \\\\")

    lines.append("\\end{tabular}\\medskip")
    lines.append(
        "\\caption{The exceptional sets $S_k$ for $0 \\leq k \\leq "
        + str(max_k)
        + "$.}"
    )
    lines.append("\\label{tab:exceptional_sets}")
    lines.append("\\end{table}")

    return "\n".join(lines)

def generate_flanking_table(max_p: int = 10000) -> str:
    """
    Generate LaTeX table of values flanked by 14 at distance 1.

    Displays values in a set-like format with k_min annotations in scriptsize.
    """
    flanked_with_data = find_flanked_values(max_p)

    lines: List[str] = []
    lines.append(f"\\newcommand\\maxconsidered{{{max_p}}}")
    lines.append("")
    lines.append("\\begin{table}[htb]")
    lines.append("\\centering")

    # Handle case where no flanked values are found
    if not flanked_with_data:
        lines.append("\\begin{align*}")
        lines.append("\\{\\}")
        lines.append("\\end{align*}")
        lines.append(
            "\\caption{No values $n = 2p$ with $p \\leq \\maxconsidered$ "
            + "such that $14$ flanks $n$ at distance~$1$ were found.}"
        )
        lines.append("\\label{tab:flanked_values}")
        lines.append("\\end{table}")
        return "\n".join(lines)

    # Build the set notation with character-based line breaking
    set_lines: List[str] = []
    current_line = ""
    current_display_length = 0  # approximate displayed character count

    for i, (val, _period, k_min) in enumerate(flanked_with_data):
        val_str = f"{val}{{\\scriptsize[{k_min}]}}"
        display_str = f"{val}[{k_min}]"

        if i < len(flanked_with_data) - 1:
            val_str += ", "
            display_str += ", "

        if current_line and current_display_length + len(display_str) > 80:
            set_lines.append(current_line.rstrip(", "))
            current_line = val_str
            current_display_length = len(display_str)
        else:
            current_line += val_str
            current_display_length += len(display_str)

    if current_line:
        set_lines.append(current_line.rstrip(", "))

    # Build the align* environment
    lines.append("\\begin{align*}")
    
    if len(set_lines) == 1:
        # Single line case
        lines.append("\\{&" + set_lines[0] + "\\}")
    else:
        # Multi-line case
        lines.append("\\{&" + set_lines[0] + ",\\\\")
        for line in set_lines[1:-1]:
            lines.append("&" + line + ",\\\\")
        lines.append("&" + set_lines[-1] + ",\\ldots\\\\")
        lines.append("&\\}")
    
    lines.append("\\end{align*}")

    lines.append(
        "\\caption{All values $n = 2p$ with $p \\leq \\maxconsidered$ "
        + "such that $14$ flanks $n$ at distance~$1$, "
        + "with first appearance $k_{\\min}$ shown in brackets. "
        + f"We found ${len(flanked_with_data)}$ such values among primes "
        + "$p \\leq \\maxconsidered$ with $p \\equiv 3 \\pmod{{4}}$.}"
    )
    lines.append("\\label{tab:flanked_values}")
    lines.append("\\end{table}")

    return "\n".join(lines)

def generate_source_appendix() -> str:
    """
    Generate a LaTeX quine of this source code for the Appendix.

    Returns:
        LaTeX code using listings package.
    """
    if hasattr(sys, "ps1"):
        source_code = "% Source code not available in interactive mode"
    else:
        try:
            current_file = os.path.abspath(__file__)
            with open(current_file, "r", encoding="utf-8") as f:
                source_code = f.read()
        except Exception:
            source_code = "% Could not read source file"

    lines: List[str] = []
    lines.append("% Add to preamble:")
    lines.append("% \\\\usepackage{listings}")
    lines.append("% \\\\usepackage{xcolor}")
    lines.append("% ")
    lines.append("% \\\\definecolor{codegreen}{rgb}{0,0.6,0}")
    lines.append("% \\\\definecolor{codegray}{rgb}{0.5,0.5,0.5}")
    lines.append("% \\\\definecolor{codepurple}{rgb}{0.58,0,0.82}")
    lines.append("% \\\\definecolor{backcolour}{rgb}{0.98,0.98,0.98}")
    lines.append("% ")
    lines.append("% \\\\lstdefinestyle{pythonstyle}{")
    lines.append("%     language=Python,")
    lines.append("%     backgroundcolor=\\\\color{backcolour},")
    lines.append("%     commentstyle=\\\\color{codegreen},")
    lines.append("%     keywordstyle=\\\\color{magenta},")
    lines.append("%     numberstyle=\\\\tiny\\\\color{codegray},")
    lines.append("%     stringstyle=\\\\color{codepurple},")
    lines.append("%     basicstyle=\\\\ttfamily\\\\footnotesize,")
    lines.append("%     breakatwhitespace=false,")
    lines.append("%     breaklines=true,")
    lines.append("%     captionpos=b,")
    lines.append("%     keepspaces=true,")
    lines.append("%     numbers=left,")
    lines.append("%     numbersep=5pt,")
    lines.append("%     showspaces=false,")
    lines.append("%     showstringspaces=false,")
    lines.append("%     showtabs=false,")
    lines.append("%     tabsize=4,")
    lines.append("%     frame=single,")
    lines.append("%     rulecolor=\\\\color{black}")
    lines.append("% }")
    lines.append("")
    lines.append("\\lstset{style=pythonstyle}")
    lines.append(
        "\\begin{lstlisting}[caption={\\texttt{Python} implementation for "
        "computing exceptional sets and values flanked by $14$, as well as "
        "this quine of the generator source code.}]"
    )
    lines.append(source_code)
    lines.append("\\end{lstlisting}")

    return "\n".join(lines)

# ============================================================================
# MAIN EXECUTION
# ============================================================================

def main() -> None:
    """
    Main function to generate all tables.
    """
    print("=" * 60)
    print("EXCEPTIONAL SETS AND FLANKING PATTERNS")
    print("=" * 60)

    # 1. Generate exceptional sets table
    print("\n1. Generating exceptional sets table...")
    exceptional_table = generate_exceptional_sets_table(max_k=35)
    with open("exceptional_sets_table.tex", "w", encoding="utf-8") as f:
        f.write(exceptional_table)
    print("   Saved to: exceptional_sets_table.tex")

    # 2. Generate flanking table
    print("\n2. Computing flanked values using theoretical criterion...")
    flanking_table = generate_flanking_table(max_p=10000)
    with open("flanked_values_table.tex", "w", encoding="utf-8") as f:
        f.write(flanking_table)
    print("   Saved to: flanked_values_table.tex")

    # 3. Generate source code quine
    print("\n3. Generating quine Appendix...")
    source_listing = generate_source_appendix()
    with open("source_code_appendix.tex", "w", encoding="utf-8") as f:
        f.write(source_listing)
    print("   Saved to: source_code_appendix.tex")

    print("\n" + "=" * 60)
    print("COMPLETE")
    print("=" * 60)
    print("\nTo include in LaTeX document:")
    print("  \\\\input{exceptional_sets_table.tex}")
    print("  \\\\input{flanked_values_table.tex}")
    print("  \\\\input{source_code_appendix.tex}")

if __name__ == "__main__":
    main()

\end{lstlisting}